\documentclass[12pt,a4paper,reqno]{amsart}
\usepackage{graphicx} 
\numberwithin{equation}{section}
\usepackage[T1, T2A]{fontenc}
\usepackage[greek, russian, english]{babel}
\usepackage[colorlinks=true, linkcolor=blue]{hyperref}
\usepackage{orcidlink}
\title[On the speed of convergence in the ergodic theorem for shift operators]{On the speed of convergence in the ergodic \\theorem for shift operators}

\date{}
\usepackage{amsfonts,amsmath,amssymb, amsthm}
\usepackage[a4paper,left=2cm,right=2cm]{geometry}
\DeclareMathOperator*{\SPAN}{span}

\DeclareMathOperator{\dist}{dist}
\DeclareMathOperator{\GL}{GL}
\DeclareMathOperator{\tr}{tr}
\DeclareMathOperator{\supp}{supp}

\newcommand{\bC}{\mathbb C}

\newcommand{\bN}{\mathbb N}

\newcommand{\bR}{\mathbb R}

\newcommand{\bT}{\mathbb T}

\newcommand{\bZ}{\mathbb Z}
\newcommand{\cH}{\mathcal H}
\newcommand{\cV}{\mathcal V}
\newcommand{\cM}{\mathcal M}
\newtheorem{thm}{Theorem}[section]
\newtheorem{lem}[thm]{Lemma}
\newtheorem{prop}[thm]{Proposition}
\newtheorem{cor}[thm]{Corollary}

\theoremstyle{definition}

\theoremstyle{definition}
\newtheorem*{thm*}{Theorem}

\usepackage{tcolorbox}

\newcommand{\sumN}{\sum_{n=0}^{N-1}}

\author[N. Chalmoukis]{N. Chalmoukis$^2$ \orcidlink{0000-0001-5210-8206}}
\address{\phantom{i}$^1$Dipartimento di Ingegneria Gestionale, dell'Informazione e
della Produzione, Universit{\`a} degli Studi di Bergamo, Viale G. Marconi 5, 24044,
Dalmine BG, Italy}
\email{biancamaria.gariboldi@unibg.it}
\email{alessandro.monguzzi@unibg.it}
\author[L. Colzani]{L. Colzani$^2$}
\address{\phantom{1}$^2$Dipartimento di Matematica e Applicazioni, Universit\`a degli
Studi di Milano Bicocca, Via R. Cozzi 55,  20125, Milano, Italy}
\email{leonardo.colzani@unimib.it}
\email{nikolaos.chalmoukis@unimib.it}
\author[B. Gariboldi]{B. Gariboldi$^{1}$ \orcidlink{0000-0001-8714-4135}}
\author[A. Monguzzi]{A. Monguzzi$^1$ \orcidlink{0000-0003-3233-5000}} 
\thanks{MSC 2020: 47A35, 37A30}
\thanks{All the authors are members of Indam--Gnampa. N. Chalmoukis and A. Monguzzi are partially supported by the Hellenic Foundation for Research and Innovation
(H.F.R.I.) under the “2nd Call for H.F.R.I. Research Projects to support Faculty Members \&
Researchers” (Project Number: 73342) and the Indam--Gnampa project CUP\textunderscore E53C22001930001. A. Monguzzi and B. Gariboldi are supported by the PRIN 2022 project "TIGRECO - TIme-varying signals on Graphs: REal and COmplex methods" funded by the MUR (Ministero dell'Università e della Ricerca), Grant\textunderscore 20227TRY8H, CUP\textunderscore F53D23002630001}
\keywords{ergodic theorem, convergence rate, shift operators, toral endomorphisms}

\begin{document}

\maketitle
\begin{abstract}
    % We consider ergodic means of shift operators and prove results on the speed of convergence of such means to their limit. More precisely, 
    Given a probability space $(X,\mu)$, a square integrable function $f$ on such space and a (unilateral or bilateral) shift operator $T$, we prove under suitable assumptions that the ergodic means $N^{-1}\sumN T^nf$ converge pointwise almost everywhere to zero with a speed of convergence which, up to a small logarithmic transgression, is essentially of the order of $N^{-1/2}$. We also provide a few applications of our results, especially in the case of shifts associated with toral endomorphisms.
\end{abstract}
\section{Introduction and main results}
Let $(X, \mu)$ be a probability space and $T$ a bounded linear operator on the Hilbert space $L^2(X,\mu)$. For $f\in L^2(X,\mu)$ consider its ergodic means 
\[ \frac{1}{N}\sum_{n=1}^{N-1}T^nf(x), \quad N \geq 1, x \in X. \] 
In this paper we study the speed of convergence of such ergodic means when $T$ is a unilateral or bilateral shift operator. Shift operators are sometimes induced by ergodic transformations. Thus, our results also cover some particular instances of von Neumann's \cite{vN} and Birkhoff's \cite{B} ergodic theorems. % In this paper we study the speed of convergence in von Neumann's \cite{vN} and Birkhoff's \cite{B} ergodic theorems for some particular ergodic systems. 
It is well-known that, in full generality, Birkhoff's and von Neumann's theorems are optimal, in the sense that the speed of convergence can indeed be arbitrarily slow, either in norm or in the sense of almost everywhere convergence (see \cite{Krengel, KP}, cf. Theorem \ref{thm:no-speed}). 
Nonetheless, scholars have been intensively investigating such problems from different perspectives and with different goals in mind. To keep track of the literature, as it often happens, is a hard task and here we recall only a few meaningful papers, apologising for the ones we omit. In \cite{FS} Furman and Shalom consider the measure-preserving and ergodic action of a locally compact group acting on a probability space and study the ergodic properties of the action along random walks on $G$. The setting described in \cite{FS} is quite different from ours, however the results obtained are similar in the spirit with the ones we obtain here (cf. \cite[Theorem 1.2]{FS} with Theorem \ref{thm:maximal-unilateral-general}). Kachurovski\u{\i}, Podvigin and coauthors have been studying the problem for the last decades from the spectral theory point of view and we refer the reader to the survey \cite{KPo2}. \begin{comment} and to the more recent works \cite{KLK, KPo3, Po, KPo4}.\end{comment}
In the same spirit of the work of Kachurovski\u{\i} and collaborators we also mention the work \cite{BAB}. Avigad and collaborators investigated the rate of convergence in \cite{AGT, AR, AI} in the sense of \emph{metastability} (see \cite{Tao}). Finally, we mention the work of Das and Yorke \cite{DY}, of Bayart, Buczolich and Heureaux \cite{BBH} and of Colzani, Gariboldi and Monguzzi \cite{C1, CGM}, who all obtain results on the speed of convergence when one considers as transformation the map $x\to x+\alpha$, which is an ergodic transformation of the $d$-dimensional torus $\mathbb T^d=\mathbb R^d/\mathbb Z^d$ whenever $\alpha=(\alpha_1,\ldots,\alpha_d)$ is an irrational vector, that is, whenever $1,\alpha_1,\ldots,\alpha_d$ are linearly independent over $\mathbb Q$. 

In order to provide some context for our results let us focus for a moment on a specific transformation, namely, the doubling map $x\mapsto 2x\mod 1$, which is a well-known ergodic transformation of the one-dimensional torus $\mathbb T$. The sum $\sumN f(2^nx)$ satisfies the central limit theorem and the law of iterated logarithm for a large class of functions. See the work of Fortet \cite{Fortet}, Kac \cite{Kac} and Maruyama \cite{Maruyama}. For subsequent extension of these results we mention, among others, the works of Aistleitner \cite{A1,A2} and refer to the references therein. More in detail, Maruyama, building upon the results of Kac, proved that if $f$ is a continuous function with vanishing mean and satisfying a H\"older condition of order $\alpha>0$, then, for almost every $x$,
\[
\limsup_{N\to+\infty}\frac{1}{\sqrt{2N\log\log(N)}}\sumN f(2^nx)= \lim_{N\to+\infty}\bigg( \frac{1}{N}\int_{\mathbb T} \Big(\sumN f(2^ny)\Big)^2\, dy\bigg)^{\frac12}.
\]
The point of view in the papers we mentioned focuses on the lacunarity of the sequence $\{2^nx\}_{n\in\mathbb N}$ and on the analogy with systems of independent random variables. In this work, instead, we take advantage of the fact that the composition operator $Tf(x)=f(2x)$ is a shift operator on $L^2(\mathbb T, dx)$ (see below for the exact definition).

Before stating our results we briefly recall some definitions following \cite{Nagy}.  Let $\cH$ be a complex separable Hilbert space endowed with the inner product $\langle\cdot,\cdot\rangle$. Let $T:\cH\to \cH$ be an isometry, that is, a bounded linear operator such that
\[
\langle Tf, Tg\rangle=\langle f,g\rangle\qquad \forall f,g\in \cH.
\]
A subspace $\cV\subseteq \cH$ is called a \emph{wandering subspace} for the isometry $T:\cH\to \cH$ if 
\[
T^m (\cV)\perp T^n (\cV)\qquad \forall m,n\in\mathbb N\cup\{0\} , m\neq n.
\]
The isometry $T:\cH\to \cH$ is a \emph{unilateral shift} if there exists a wandering subspace  $\cV\subseteq \cH$ for $T$ such that 
\[
\cH=\bigoplus_{k\in\mathbb N\cup\{0\}}T^k(\cV).
\]
In this case we say that the subspace $\cV$ is a \emph{generating} wandering subspace for $T$. Notice that 
\[
\cV=\cH\ominus T(\cH).
\]

Unilateral shifts are ubiquitous in operator theory. One  reason for this is provided by Wold's decomposition theorem.
\begin{thm*}[Wold decomposition]
\textit{Let $T:\cH\to\cH$ be an isometry. Then, 
\[
\cH=\cM\oplus\cM^\perp
\]
where $\cM$ and $\cM^\perp$ are invariant under $T$, $T:\cM\to\cM$ is a unilateral shift and $T:\cM^\perp\to\cM^\perp$ is a unitary operator.  Such decomposition is uniquely determined and it holds
\[
\cM=\bigoplus_{k\in\mathbb N \cup \{0\} } T^k(\cH\ominus T(\cH)),\qquad\cM^\perp=\bigcap_{k\in\mathbb N}T^k(\cH). 
\]}
\end{thm*}
Similarly to unilateral shifts, it is possible to define bilateral shifts. %Let now $T:\cH\to\cH$ be a unitary operator.  
A subspace $\cV\subseteq \cH$ is called a wandering subspace for the \emph{unitary} operator $T:\cH\to \cH$ if 
\[
T^m (\cV)\perp T^n (\cV)\qquad \forall m,n\in\mathbb Z , m\neq n
\]
and $T:\cH\to \cH$ is a \emph{bilateral shift} if there exists a generating wandering subspace  $\cV\subseteq \cH$ such that 
\[
\cH=\bigoplus_{k\in\mathbb Z}T^k(\cV).
\]
Notice that for bilateral shifts the generating wandering subspace is not uniquely determined.

If $T:\cH\to\cH$ is a shift, then $\cH$ admits an orthonormal basis of the form $\{\varphi_{j,k}\}_{j\in\mathbb X, k\in\mathbb Y}$, where $\mathbb X\subseteq \mathbb N$ and $\mathbb Y$ is either $\mathbb N\cup\{0\}$ or $\mathbb Z$ depending on $T$ being a unilateral or bilateral shift, such that $\{\varphi_{j,k}\}_{j\in\mathbb X}$ is an orthonormal basis for $T^k(\mathcal V)$ for every $k\in\mathbb Y$ and such that, for every fixed $k\in\mathbb Y$, it holds 
\[
T\varphi_{j,k}=\varphi_{j,k+1}.
\]
From now on when we say that the isometry $T:\cH\to\cH$ is a shift we mean that $T$ could be either a unilateral or a bilateral shift. However, the reader has to keep in mind that whenever $T$ is intended as a bilateral shift then $T$ is not only an isometry, but a unitary operator as well.

We now introduce the general setting in which our results take place. 
We will assume the following:
\begin{enumerate}
\item[$(i)$] $\cH$ is a Hilbert space and  $T:\cH\to\cH$ is an isometry ;
\item[$(ii)$] $\cH= \cM\oplus \cM^\perp$ where $T|_{\cM}: \cM\to \cM$ is a shift (bilateral or unilateral) and $T|_{\cM^{\perp}}: \cM^\perp\to \cM^\perp$ is the identity operator; i.e., we are considering isometries whose unitary part in the Wold decomposition is the identity operator.
\item[$(iii)$] $\mathcal V$ is a generating wandering subspace for $T|_\cM$ and $\Pi_{\cM^\perp}$ and $\Pi_k$ are the orthogonal projections from $\cH$ onto $\cM^\perp$ and $T^{k}(\cV)$ respectively. Here $k$ varies either in $\mathbb N\cup\{0\}$ or $\mathbb Z$ accordingly with the fact that $T$ is a unilateral or a bilateral shift.
\end{enumerate}
The following theorem is implicit in the existing literature, but we could not find a precise reference. In particular, when $T$ is a shift such that $\dim(\mathcal{V}) = + \infty$ the theorem is proved in \cite{Krengel} and \cite{KP}. Anyhow, a short proof will be included for the reader's convenience. 
\begin{thm}\label{thm:no-speed}

With the notation above, for every positive vanishing sequence $\varepsilon_n\to 0$ as $n\to+\infty$, there exists $f\in\cH$ such that
 \[
 \limsup_{N\to +\infty} \varepsilon_N^{-1} \bigg\|\frac{1}{N}\sumN T^nf-\Pi_{\cM^\perp}f\bigg\|_\cH=+\infty.
 \]
\end{thm}

Despite the negative result in the previous theorem, it is possible to give some positive results on the speed convergence under appropriate assumptions on the operator and on the functions. The following result is no surprising and we include it for the sake of completeness.   
\begin{thm}\label{thm:unilateral-norm-general}
With the notation above, \begin{align*}
\bigg\|\frac{1}{N}&\sumN T^nf-\Pi_{\cM^\perp}f\bigg\|_\cH\leq \frac{1}{\sqrt{N}}\sum_{k}
\|\Pi_k f\|_{\cH}.
\end{align*}
Moreover, the rate of convergence $1/\sqrt{N}$ is sharp.
\end{thm}

The next theorem is our first main one. We obtain a result on the pointwise speed of convergence and the boundedness of a maximal function.

\begin{thm}\label{thm:maximal-unilateral-general}
With the notation above, assume that $\mathcal H$ is the function space $L^2_\mu:= L^2(X,d\mu)$ where $(X, \mu)$ is a probability space, and that $\varepsilon:\mathbb R_+\to \mathbb R_+$ is a positive decreasing function. %such that $\sup_n \varepsilon(n-1)/\varepsilon(n)$ is bounded. 
Define the maximal operator
    \[
    Sf (x) = \sup_{N\geq 1} N\varepsilon(N)\bigg| \dfrac{1}{N} \sumN T^n f (x) - \Pi_{\cM^\perp}f(x) \bigg|.
    \]    
 Then,   there exists a positive constant $c$ such that
    \begin{equation}\label{main-maximal-estimate}
    \|S f\|_{L^2(X,\mu)}\leq c \bigg( \sum_{n=0}^{+\infty} \varepsilon^2(n) \log^2(n+2) \bigg)^{\frac12}\sum_{k} \|\Pi_k f\|_{L^2_{\mu}}.
    \end{equation}
Moreover, if 
\begin{equation}\label{finite_condition} \sum_{n=0}^{+\infty} \varepsilon^2(n) \log^2(n+2) < +\infty \quad \text{and}  \quad  \sum_{k} \|\Pi_k f\|_{L^2_{\mu}} <+\infty, \end{equation}
then, for $\mu$-almost every $x$,
    \begin{equation}\label{2}
    \lim_{N\to+\infty} N\varepsilon(N)\bigg|\frac{1}{N} \sumN T^n f (x) - \Pi_{\cM^\perp}f(x) \bigg|= 0.
    \end{equation}
   \end{thm}
The following are two immediate applications of the above theorem. 

\begin{cor}\label{cor:baker}
Let $B:[0,1)^2\to [0,1)^2$ be the baker's transformation defined by
\[
B(x,y)=\begin{cases}
(2x, \frac{y}{2}) &  {\textrm{if }} 0\leq x<\frac 12\\
(2x-1, \frac{y}{2}+\frac{1}{2}) &  {\textrm{if }} \frac12\leq x<1.
\end{cases}
\]
Assume that $f$ has an absolutely convergent expansion with respect to the product Walsh system on the square $[0,1)^2$.
Then, for every $\eta>0$ and for almost every $x$,
\[
\lim_{N\to+\infty} \frac{\sqrt{N}}{(\log(1+N))^{\frac32+\eta}}\bigg|\frac{1}{N}\sumN f(B^n x)-\int_{\mathbb T^d} f(y)\, dy\bigg|=0.
\]
\end{cor}

\begin{cor}\label{cor:Laguerre} Let $T$ be the operator \[Tf(x)=f(x)-\int_0^x f(y)dy\] defined on the Hilbert space $L^2(\bR_+, e^{-x}\, dx)$ and let  $\{L_n\}_{n\in\mathbb N}$ be the system of Laguerre polynomials. Assume that the Laguerre coefficients of $f$ are absolutely summable. 
Then, for every $\eta>0$ and for almost every $x$,
\[
\lim_{N\to+\infty}\frac{\sqrt{N}}{(\log(1+N))^{\frac32+\eta}}\bigg| \frac{1}{N}\sum_{n=0}^{N-1} T^n f(x) \bigg|=0.
\]
\end{cor}

Our last theorem is about ergodic means associated with the endomorphisms of the $2$-dimensional torus $\mathbb{T}^2=\mathbb{R}^2/\mathbb{Z}^2$ and the classical trigonometric expansion. We prove that it is enough to require a mild summability condition with respect to a logarithmic weight on the Fourier coefficients of a function to gain a speed of convergence essentially of order $N^{-\frac 12}$ for the ergodic means.

\begin{thm}\label{thm:2x2matrix-complete}
Let $A$ be a $2\times 2$ integer matrix such that $\det(A)\neq 0$ and no eigenvalue of $A$ is a root of unity. Assume that $f\in L^2(\mathbb T^2, dx)$ has the trigonometric expansion 
\[ 
f(x) = \sum_{\xi \in \bZ^2} \widehat{f}(\xi) e^{2\pi i x \xi } 
\]
and that, for some $\delta>0$, 
\begin{equation} \label{condition_Fourier}
\sum_{\xi\in\mathbb Z^2}(\log(1+\vert \xi \vert ))^{1+\delta} |\widehat{f}(\xi)|^2<+\infty. 
\end{equation}
Then, for every $\eta>0$ and for almost every $x\in\mathbb T^2$,
\[
\lim_{N\to+\infty} \frac{\sqrt{N}}{(\log(1+N))^{\frac32+\eta}}\bigg|\frac1{N}\sumN f(A^nx)-\int_{\mathbb T^2}f(y)\, dy\bigg|=0.
\]
\end{thm}
We point out that, in the above theorem, $A$ has no eigenvalues which are not roots of unity if and only if $A$ is an ergodic matrix \cite[Corollary 2.2]{EW}. Therefore, the above theorem guarantees a speed a convergence for the ergodic means of  a large class of functions for a particular instance of Birkhoff's ergodic theorem.  Condition \eqref{condition_Fourier} is satisfied, for instance, by functions in any fractional Sobolev space. A more general sufficient condition in terms of the $L^2$ integral modulus of continuity will be given in Proposition \ref{prop:continuity} below.

The situation in dimension $d>2$ seems to be more complicated. Nonetheless, we prove the following partial result, which is a corollary of Theorem \ref{thm:maximal-unilateral-general}.

\begin{cor}\label{cor:A}
Let $A$ be a $d\times d$ matrix with integer coefficients and $\det(A)\neq 0$. Suppose there exists a set $\mathcal{E} \subseteq\bZ^d\setminus \{ 0\} $ such that the subspace of $L^2_0(\bT^d,dx)$
\[
\mathcal{V}_\mathcal{E}: = \{ f \in L^2_0(\bT^d, dx) : \supp(\widehat{f}) \subseteq \mathcal{E} \}
\]
is a generating wandering subspace for the operator $ T_A f = f \circ A $. Suppose that there exist $c>0, q>1 $, such that for all $\xi \in \mathcal{E}$ and $k \in \mathbb{Y}$ (where $\mathbb{Y}$ is either $\mathbb{N} \cup \{ 0 \} $ or $\mathbb{Z}$ depending on whether $T_A$ is a unilateral or bilateral shift),
\begin{equation}\label{svp} |A^k \xi | \geq c q^{|k|}. \end{equation}
Assume that $f\in L^2(\mathbb T^2, dx)$ has the trigonometric expansion 
\[ 
f(x) = \sum_{\xi \in \bZ^2} \widehat{f}(\xi) e^{2\pi i x \xi } 
\]
and that, for some $\delta>0$, 
\begin{equation} \label{condition_Fourier_2}
\sum_{\xi\in\mathbb Z^2}(\log(1+\vert \xi \vert ))^{1+\delta} |\widehat{f}(\xi)|^2<+\infty. 
\end{equation}
Then, for every $\eta>0$ and for almost every $x\in\mathbb T^2$,
\[
\lim_{N\to+\infty} \frac{\sqrt{N}}{(\log(1+N))^{\frac32+\eta}}\bigg|\frac1{N}\sumN f(A^nx)-\int_{\mathbb T^2}f(y)\, dy\bigg|=0.
\]
\end{cor}
Assumption \eqref{svp} is satisfied, for instance, whenever $A$ is an expansive matrix, i.e., whenever there exists $q>1$ such that $|Ax|\geq q |x| $ for all $x \in \bR^d$.

We should also mention that in the literature there exist theorems of flavor similar to Theorem \ref{thm:2x2matrix-complete}. For example in \cite[Theorem 1.2]{Lobbe} the author proves the law of the iterated logarithm for averages of the form 
\[\frac 1N\sum_{n=0}^{N-1}f(M_n x) \]
where $(M_n)_{n\geq 1}$ is a sequence of integer matrices satisfies a strong Hadamard type condition \cite[Condition (1.4)]{Lobbe} and $f$ is a function of finite Hardy-Krause total variation. Although our theorem gives less precise asymptotic information than the law of the iterated logarithm, our assumptions are much less stringent. If $A$ is a matrix as in Theorem \ref{thm:2x2matrix-complete}, then the sequence $M_n:=A^n$ does not in general satisfy \cite[Condition (1.4)]{Lobbe} and functions satisfying \eqref{condition_Fourier} can be quite rough.

\section{Proof of Theorems \ref{thm:no-speed}, \ref{thm:unilateral-norm-general} and  \ref{thm:maximal-unilateral-general} and of Corollaries \ref{cor:baker} and \ref{cor:Laguerre}}\label{sec:main_results}
 The proof of Theorem \ref{thm:no-speed} is straightforward.
\begin{proof}[Proof of Theorem \ref{thm:no-speed}]
Since $T$ has operator norm  $1$, the averaging operator $U_N:=\frac{1}{N}\sumN T^n$ has operator norm at most $1$. Furthermore, the norm is at least $1$, as it can be seen by testing the operator $U_N$ on the functions $f_H=\sum_{k=0}^{H}\varphi_{j,k}$ and letting $H\to+\infty$. Here $\{\varphi_{j,k}\}_{j,k}$ is an orthonormal basis associated to the shift $T$. Therefore, the family of operators $\{\varepsilon^{-1}_N U_N\}_{N}$ is not uniformly bounded in the operator norm. Hence, by the Banach-Steinhaus uniform boundedness principle, there exists $f\in \mathcal M\subseteq\cH$ such that 
\[
\limsup_{N\to+\infty} \varepsilon_N^{-1}\bigg\| \frac{1}{N}\sumN T^nf\bigg\|_{\cH}=+\infty.
\]
\end{proof}
As mentioned, Theorem \ref{thm:unilateral-norm-general}  can also be proved using the unitary equivalence with the shift operator on vector valued Hardy spaces in the unit disc. However, for the sake of completeness, we provide here a direct proof.

\begin{proof}[Proof of Theorem \ref{thm:unilateral-norm-general}]
The proof for unilateral or bilateral shifts is the same.  Let $T:\cM\to \cM$ be a bilateral shift. Then there exists a generating wandering subspace $\mathcal V$ such that
\[
\cH=\cM\oplus \cM^\perp=\bigg(\bigoplus_{k\in\mathbb Z}T^k\big(\cV\big)\bigg)\oplus \cM^\perp.
\]
Let $\{\varphi_{j,k}\}_{j\in\mathbb X,k\in\mathbb Z}$ be an orthonormal basis of $\cM$ associated to $T$. Without losing generality, we assume that $f$ has only finitely many nonzero coefficients $\{\widehat f(j,k)\}_{j\in\mathbb X,k\in\mathbb Z}$ with respect to the orthonormal basis $\{\varphi_{j,k}\}_{j\in\mathbb X,k\in\mathbb Z}$. Since $T$ acts as the identity on $\cM^\perp$, we have 
\begin{align*}
 \frac{1}{N} \sumN T^n f-\Pi_{\cM^\perp}f &=  
  \frac{1}{N} \sumN \sum_{j,k}\widehat f(j,k)\varphi_{j,k+n}=\frac{1}{\sqrt N}\sum_{j,k}\widehat f(j,k)\Psi_{j,k}(N),
\end{align*}
where we have set
\[
\Psi_{j,k}(N)=\frac{1}{\sqrt N} \sumN \varphi_{j,k+n}.
\]
It can be readily checked that $\{\Psi_{j,k}(N)\}_{j\in\mathbb X}$ is an orthonormal system for every fixed $k\in\mathbb Z$. Hence, by Parserval's identity, \begin{align*}
\bigg\| \frac{1}{N}\sumN T^nf-\Pi_{\cM^\perp} f\bigg\|&=\bigg\|\frac{1}{\sqrt N}\sum_{j,k}\widehat f(j,k)\Psi_{j,k}(N)\bigg\|\\
&\leq \frac{1}{\sqrt N}\sum_{k\in\mathbb Z}     \bigg\|\sum_{j\in\mathbb X}\widehat f(j,k) \Psi_{j,k}(N)\bigg\|\\
&= \frac{1}{\sqrt{N}} \sum_{k\in\mathbb Z} \bigg(\sum_{j\in\mathbb X}|\widehat f(j,k)|^2\bigg)^{\frac12}= \frac{1}{\sqrt{N}} \sum_{k\in\mathbb Z} \|\Pi_kf\|.
\end{align*}

Finally, observe that if $f=\Pi_{k}f$ for a single $k$, then all the above inequalities actually are identities. Hence, the theorem is sharp.
\end{proof}

The proof of Theorem \ref{thm:maximal-unilateral-general} is in principle similar to the proof of Theorem \ref{thm:unilateral-norm-general}. The main ingredient is the Rademacher--Menshov theorem, which we now recall. 
\begin{thm}[Rademacher--Menshov]
    There exists an absolute positive constant $C$ such that for every positive measure space $ (X,\mu)$ and every orthogonal system $f_0, f_1 \dots  $ in $L^2(X,\mu)$, the maximal function 
    \[ \mathcal{M}(x): = \sup_{ k \geq 0 } \bigg| \sum_{n=0}^k f_n(x) \bigg| \]
    satisfies the estimate
    \[ \Vert \mathcal{M} \Vert _{L^2(X,\mu)} \leq C  \bigg( \sum_{n=0}^{+\infty}  \log^2(n+2) \Vert f_n \Vert^2_{L^2(X,\mu)} \bigg)^{\frac12}. \]
\end{thm}

It is important to emphasize that the constant $C$ in the above theorem is absolute and we refer the reader to \cite{M} for a discussion on this. 

Recall also the next lemma by Kronecker, which is an application of Abel's summation by parts formula.

\begin{lem} \label{}
    Suppose that $a_n$ is a sequence of complex numbers such that $\sum_{n=1}^\infty a_n$, exists and is finite. Assume also that $b_n$ is a  nondecreasing sequence of  positive numbers tending to infinity. Then,
    \[ \lim_{N \to \infty} \frac{1}{b_N} \sum_{n=0}^{N-1}b_n a_n = 0. \]
\end{lem}

\begin{proof}[Proof of Theorem \ref{thm:maximal-unilateral-general}]
We assume again that $T$ is a bilateral shift. The proof for the unilateral case is the same. To simplify the notation we also assume that $f$ is in $\cM$, so that $\Pi_{\cM^\perp}f=0$. Finally, assume that $f$ has only finitely many nonzero Fourier coefficients. Then,
\begin{align*}
 N\varepsilon(N)\bigg(\frac{1}{N} \sumN T^{n}f(x)- \Pi_{\cM^\perp}f(x)\bigg)=  \varepsilon(N)\sumN T^n f (x).
 \end{align*}

We derive both \eqref{main-maximal-estimate} and \eqref{2} from the boundedness of an auxiliary maximal function. Let $\varepsilon : [0,+\infty) \to \mathbb{R}$, not necessarily decreasing, and define 
 \[  \Tilde{S}f(x) := \sup_{N \geq 1} \Big| \sum_{n=0}^{N-1} \varepsilon(n) T^nf(x) \Big|. \]
We have
\begin{align*}
\sumN \varepsilon(n) T^n f(x)&=\sumN \varepsilon(n) \sum_{j\in\mathbb X,k\in\mathbb Z}\widehat f(j,k)\varphi_{j,k+n}(x)= \sum_{k\in\mathbb Z} A(k)\sumN \varepsilon(n)\Phi(k,n,x)
\end{align*}
where we have set
\[
A(k)=\|\Pi_k f\|_{L^2_\mu}=\bigg(\sum_{j\in\mathbb X}|\widehat f(j,k)|^2\bigg)^{\frac 12},\qquad\Phi(k,n,x)=\frac{1}{A(k)}\sum_{j\in\mathbb X}\widehat f(j,k)\varphi_{j,k+n}(x).
\]
Then,  
\begin{align*}
\begin{split}
\sup_{N\geq 1}\bigg|&  \sumN  \varepsilon(n)T^n f (x)\bigg|= \sup_{N\geq 1}\bigg|\sum_{k\in\mathbb Z} A(k)\sumN \varepsilon(n)\Phi(k,n,x)\bigg|.
\end{split}
\end{align*}
In the above formula we simply omit the terms such that $A(k)=0$. It may be promptly verified  that $\{\Phi(k,n,x)\}_{n=0}^{N-1}$ is an orthonormal system for every fixed $k\in\mathbb Z$ and $N\in\mathbb N$.  
Hence, by means of the Rademacher--Menshov theorem,
\begin{align} \label{max_ineq}
\Vert \Tilde{S}f \Vert_{L^2_\mu}& = \bigg(\int_{X}\bigg(\sup_{N\geq 1}\bigg|
\sumN  \varepsilon(n)T^n f (x)\bigg|\bigg)^2\, d\mu(x)\bigg)^{\frac 12}\nonumber \\
&\leq \sum_{k\in\mathbb Z} A(k) \bigg(\int_{X}\bigg(\sup_{N\geq 1}\bigg|\sumN \varepsilon(n)\Phi(k,n,x)\bigg|\bigg)^2\, d\mu(x)\bigg)^{\frac 12}\nonumber \\
&\leq c \bigg(\sum_{n\in\mathbb N}\varepsilon^2(n) \log^2(n+2)\bigg)^{\frac12}\sum_{k\in\mathbb Z} \| \Pi_k f\|_{L^2_\mu}.
\end{align}
Now a standard argument, as in \cite[p. 190]{Z2003}, shows that inequality \eqref{max_ineq} implies that the series $\sum_{n=0}^\infty \varepsilon(n)T^nf(x)$ converges $\mu$-a.e. Moreover, restricting to a positive decreasing $\varepsilon$, we apply Kronecker's lemma with $a_n= \varepsilon(n) T^nf(x)$, $ b_n =\varepsilon^{-1}(n)$ and we have that 
\[ \lim_{N\to \infty} \varepsilon(N)\Big| \sum_{n=0}^{N-1}T^nf(x)\Big| = 0, \quad \mu \text{-a.e.,}\]
which proves \eqref{2}.

In order to prove \eqref{main-maximal-estimate}, assume again that $\varepsilon$ is positive and decreasing. Then, by Abel's summation by parts,

\begin{align*} 
\varepsilon(N)\sumN T^nf(x)&=    \frac{ \varepsilon(N)}{\varepsilon(N-1)}\sumN \varepsilon(n) T^nf(x)\\
&\quad\quad- \varepsilon(N)\sum_{j=0}^{N-2}\bigg(\sum_{n=0}^{j}\varepsilon(n) T^nf(x)\bigg)\bigg(\frac{1}{\varepsilon(j+1)}-\frac{1}{\varepsilon(j)}\bigg).
\end{align*}
Hence, 
\begin{align}\label{eq:max-comp}
\begin{split}
Sf(x) = \sup_{N\geq 1}\varepsilon(N) \bigg|\sumN T^{n}f(x)\bigg|\leq 2\sup_{N\geq 1}\bigg|\sumN \varepsilon(n) T^{n}f(x)\bigg| =2 \, \Tilde{S}f(x).
\end{split}
\end{align}
This, together with \eqref{max_ineq} proves \eqref{main-maximal-estimate}.
\end{proof} 

We conclude the section showing that the hypothesis of Theorem \ref{thm:maximal-unilateral-general} are satisfied in the setting of Corollary \ref{cor:baker} and Corollary \ref{cor:Laguerre}.
\begin{proof}[Proof of Corollary \ref{cor:baker}]
One can
verify that the composition operator $T_Bf(x,y)=f(B(x,y))$ is a bilateral shift with respect to the product Walsh system on the square $[0,1)^2$, whose definition we now recall. Let $r_k$ be the one-dimensional $k$-th Rademacher function
\[
r_k(x)=\textrm{sgn}\big(\sin(2^{k}\pi x)\big),\qquad k\in\mathbb N, x\in[0,1).
\]
On the unit square $[0,1)^2$ define the function
\[
R_{k}(x,y):=\begin{cases}
r_{k+1}(x) & k=0, 1,2,\ldots \\
r_{|k|}(y) & k=-1,-2,\ldots \end{cases}
\]
and for every set of integers $k_1<k_2<\ldots< k_n$ define
\[
W_{k_1k_2\cdots k_n}(x,y)= R_{k_1}(x,y)\cdots R_{k_n}(x,y).
\]
Then,
\[
L^2_0([0,1)^2)=\overline{\SPAN}\bigg\{W_{k_1k_2\cdots k_n}: k_1<k_2<\ldots<k_n, k_j\in\mathbb Z, n\in\mathbb N \bigg\},
\]
where $L^2_0([0,1)^2)$ is the subspace of $L^2([0,1)^2)$ consisting of functions with vanishing mean.
One can verify that\[
T(W_{k_1k_2\cdots k_n})= W_{(k_1+1)(k_2+1)\cdots(k_n+1)}.
\]
Hence, the transformation $T$ is a bilateral shift on $L^2_0([0,1)^2)$ with a generating wandering subspace given by
\[
\cV= \overline{\SPAN}\bigg\{W_{1 k_2\cdots k_n}, 1<k_2<k_3<\ldots k_n, k_j\in\mathbb Z, n\in\mathbb N\bigg\}.
\]
Then, Theorem \ref{thm:maximal-unilateral-general} applies.
\end{proof}

\begin{proof}[Proof of Corollary \ref{cor:Laguerre}] Recall the definition of Laguerre polynomials $\{L_n\}_{n\in\mathbb N}$,
\[
L_n(x)=\frac{e^x}{n!}\frac{d^n}{dx^n}(e^{-x}x^n)=\sum_{k=0}^n \binom{n}{k}\frac{(-1)^k}{k!}x^k.
\]
This family of polynomials is an orthonormal basis for the Hilbert space $L^2(\mathbb R_+, e^{-x} dx)$. As observed by Von Neumann \cite{vN2} (see also Brown and Halmos \cite[p. 135]{Brown1965}), the operator 
\[
Tf(x) = f(x) - \int_0^x f(y)dy \] 
is the unilateral shift with respect to the Laguerre basis of $L^2(\mathbb R_+, e^{-x} dx)$. Indeed,
\[TL_n(x)=\sum_{k=0}^n \binom{n}{k}\frac{(-1)^k}{k!}x^k+\sum_{k=0}^n \binom{n}{k}\frac{(-1)^{k+1}}{(k+1)!}x^{k+1}=\sum_{k=0}^{n+1} \binom{n+1}{k}\frac{(-1)^k}{k!}x^k=L_{n+1}(x).\]
Hence,  Theorem \ref{thm:maximal-unilateral-general} applies.
\end{proof}

\section{Speed of convergence for toral endomorphisms}
\label{section-mat-2}

Before actually proving Theorem \ref{thm:2x2matrix-complete} and Corollary \ref{cor:A}, we make some preliminary observations. If in Theorem $\ref{thm:2x2matrix-complete}$ we choose a matrix $A$ with $|\det A|>1$, then the operator $T_A f= f\circ A$ is a unilateral shift on $L^2_0(\mathbb T^2)$, the space of square integrable functions with vanishing means. This is proved, e.g., in \cite{K}, but it will also follow from the proof of Lemma \ref{lem:detA_geq_1} below. If, on the other hand, $|\det A|=1$, then $T_A$ is a bilateral shift on $L_0^2(\mathbb T^2)$. A generating wandering subspace for $T_A$ can be constructed as follows. Let us consider equivalence relation on $\mathbb Z^2\setminus\{0\}$ defined by the orbits of $A^*$, i.e.
\begin{equation*}
\xi\sim \mu \iff \exists k\in\mathbb Z : A^{*k}\xi =\mu.
\end{equation*}
Let now $\mathcal E$ be the set containing of representative from each equivalence class of $\mathbb Z^2\setminus\{0\}/ \sim$. A generating wandering subspace $\mathcal {V}_{\mathcal E}$ for $T_A$ is then given by
\begin{equation}\label{eq:generating_wandering}
\mathcal V_{\mathcal E}=\{f\in\ L^2(\mathbb T^d):\supp\widehat  f\subseteq \mathcal E\}.
\end{equation}
The proof of Theorem \ref{thm:2x2matrix-complete} will follow from a series of preparatory results. In particular, we deal with  the cases $|\det A|>1$ and $|\det A|=1$ in different ways. In this latter case we will have to be more careful in constructing a generating wandering subspace $\mathcal V_{\mathcal E}$, which we recall is not unique for bilateral shifts.
\subsection{Proof of Theorem \ref{thm:2x2matrix-complete}: \underline{case $|\det A|=1$}} Let $\tr(A)$ be the trace of the matrix $A$. Observe that if $\det(A)=1$  the eigenvalues of $A$ are given by
\[
\frac{\tr(A)\pm \sqrt{\tr^2(A)-4}}{2}.
\]
Since no eigenvalue of $A$ is a root of unity by hypothesis, we can assume that $|\tr(A)|>2$. Otherwise, that is, if $\tr(A)=0,\pm 1,\pm 2$, it can be checked by hand that the eigenvalues of $A$ are roots of unity and in this case Birkhoff's theorem would not apply since the matrix $A$ would not be ergodic (see \cite{K}). If $\det(A)=-1$, then the eigenvalues of $A$ are given by
\[
\frac{\tr(A)\pm \sqrt{\tr^2(A)+4}}{2}.
\]
Notice that these are roots of unity if and only if $\tr(A)=0$. In all remaining cases we have two distinct eigenvalues $\lambda, \lambda^{-1} \in \bR $ and, without loss of generality, we can assume that $0<|\lambda|^{-1}<1<|\lambda|$. We now take advantage of this to define a suitable generating wandering subspace for the bilateral shift $T_A$. Let $S\in \GL_2(\bR)$ be such that 
\[ A = S^{-1} \begin{bmatrix}
    \lambda^{-1} & 0 \\ 0 & \lambda
\end{bmatrix} S= S^{-1}DS .
\] 
Let $\mathcal E\subseteq \mathbb Z^2\backslash\{0\}$ be such that it contains exactly one element from each orbit of the action of $A$ on $\mathbb Z^2$. We choose such element as follows. Define $|\xi|_\infty=|(\xi_1,\xi_2)|_\infty=\max\{|\xi_1|,|\xi_2|\}$. Let $\mathcal O$ an orbit of $A$
in $\mathbb{Z}^2\setminus \{ 0\}$ and consider the set $S \mathcal{O}$. Then, we choose $\xi \in \mathcal{O} $ such that $S\xi$ has the minimal $ | \cdot |_\infty $ norm. Equivalently, for all $k\in \bZ$, we have that
\begin{equation}\label{eq_E}
   |S A^k \xi |_\infty  = |D^k S \xi |_\infty \geq | S \xi|_\infty.
\end{equation}
Then, a generating wandering subspace for $T_A$ is defined as in \eqref{eq:generating_wandering}. 

Using the notation above, we prove the following.

\begin{lem}\label{lem:detA_1}
Let $A$ be a $2\times2$ integer matrix such that $|\det A|=1$ and no eigenvalues of $A$ is a root of unity. Let $\mathcal E$ be defined as above. Then, there exist constants $c>0$ and $q>1$ such that, for every $k\in\mathbb Z$,
\[
\min\{|A^k\xi|: \xi \in  \mathcal E \}\geq c q^{|k|}.
\]
\end{lem}

\begin{proof}
Assume that $\det A=1$; the case $\det A=-1$ is similar. Since for every $\xi \in \bZ^2\setminus \{ 0 \}$ and $k\in \bZ$ it holds that
\[
|A^k\xi|=|S^{-1} D^k S \xi|\geq \|S\|^{-1} |D^k S\xi|,
\]
and all norms in a finite dimensional vector space are equivalent, it suffices to show that 
 there exist $c>0$, $q>1$
such that $| D^{k}S\xi |_{\infty }\geq cq ^{|k| }$ for every $\xi \in \mathcal E $. Let $\eta=(\eta_1,\eta_2)=S\xi$ where $\xi$ is in $\mathcal E$ and let $\lambda^{-1},\lambda$ the two real eigenvalues of $A$ with $|\lambda|>1$.
Then, 
\[
|D^k\eta|_{\infty}=|(\lambda^{-1}\eta_1,\lambda^k \eta_2)|_{\infty}\geq |\lambda|^k |\eta_2|\geq \min\{|\eta_1|,|\eta_2|\}|\lambda|^k
\]
and, similarly,
\[
|D^k\eta|_{\infty}=|(\lambda^{-1}\eta_1,\lambda^k \eta_2)|_{\infty}\geq |\lambda|^{-k} |\eta_1|\geq \min\{|\eta_1|,|\eta_2|\}|\lambda|^{-k}.
\]
Hence,
\begin{equation}\label{Dk_below}
|D^k\eta|_\infty\geq \min\{|\eta_1|,|\eta_2|\}|\lambda|^{|k|}.
\end{equation}
The conclusion will follows once we prove that $\min\{|\eta_1|,|\eta_2|\}$ is bounded from below uniformly for $\eta=(\eta_1,\eta_2)$ in $\mathcal E$. But this is true because of the following. If $|\eta_2|\leq |\eta_1|$,  by the definition of $\mathcal E$, 
\[
|(\eta_1,\eta_2)|_\infty \leq |D(\eta_1,\eta_2)|_{\infty}=|(\lambda^{-1}\eta_1, \lambda \eta_2)|_\infty=|\lambda| |\eta_2|.
\]
The last identity holds since if $\left\vert \left( \lambda ^{-1}\eta _{1},\lambda \eta _{2}\right)
\right\vert _{\infty }=|\lambda^{-1}||\eta_1|$, then we would have $|\eta_1|\leq |\lambda^{-1}||\eta_1|$, which is a contradiction since $|\lambda^{-1}|<1$ and $\eta_1 \neq 0$. 
Similarly, if $\left\vert \eta _{1}\right\vert \leq \left\vert \eta _{2}\right\vert $,
\[
|(\eta_1,\eta_2)|_\infty \leq |D^{-1}(\eta_1,\eta_2)|_{\infty}=|(\lambda\eta_1, \lambda^{-1} \eta_2)|_\infty=|\lambda| |\eta_1|.
\]
Hence, $|(\eta_1,\eta_2)|_\infty\leq |\lambda| \min\{|\eta_1|,|\eta_2|\}$, that is, $|\eta_1|$ and $|\eta_2|$ are comparable. Therefore, by \eqref{Dk_below},
\[
|D^k\eta|_\infty\geq \min\{|\eta_1|,|\eta_2|\}|\lambda|^{|k|}\geq |\eta|_\infty |\lambda|^{k-1}\geq c|\lambda|^k
\]
for some positive constant $c$. This follows from the fact that $\eta\in S\mathbb Z^2\backslash\{0\}$.
\end{proof}

We now conclude the proof of Theorem \ref{thm:2x2matrix-complete} in the case $|\det A|=1$. As observed at the beginning of Section \ref{section-mat-2}, the operator $T_A f(x)=f(Ax)$ is a bilateral shift on $L^2_0(\mathbb T^2)$ with a generating subspace given by $\mathcal V_{\mathcal E}$ as in \eqref{eq:generating_wandering} where $\mathcal E$ is defined by means of the property $\eqref{eq_E}$. 
Hence, Theorem \ref{thm:maximal-unilateral-general} applies and, in particular, it applies with $\varepsilon(n)=(n+1)^{-\frac12}(\log(2+n))^{-\frac32- \eta}$ for any $\eta>0$. 

Set now $\mathcal{F}_k :=(A^{*})^k\mathcal E$. Observe that $A$ satisfies the hypothesis of Lemma \ref{lem:detA_1} if and only if $A^*$ does. Hence, by such lemma, there exist constants $c>0$ and $q>1$ such that, for every $k\in\mathbb Z$,
\[
\min\{|\xi|:\xi\in (A^*)^k\mathcal E\backslash\{0\}\}\geq cq^{|k|}. 
\]

Hence, for every positive increasing function $\nu$ and $f$ satisfying \eqref{condition_Fourier}, one has
\begin{align*}
\sum_{k\in\mathbb Z \cup \{ 0\} }\|\Pi_k f\|_{L^2}&=\sum_{k\in\mathbb Z \cup \{ 0\}}\bigg(\sum_{\xi \in \mathcal F_k }  |\widehat f(\xi)|^2\bigg)^{\frac12}\\
&\leq\bigg(\sum_{k\in\mathbb Z \cup \{ 0\}} \nu^{-2}(k)\bigg)^{\frac 12}\bigg(\sum_{k\in\mathbb Z \cup \{ 0\}}\nu^2(k)\sum_{\xi \in \mathcal{F}_k}|\widehat f(\xi)|^2\bigg)^{\frac12} \\
&\leq \bigg(\sum_{k\in\mathbb Z \cup \{ 0\}} \nu^{-2}(k)\bigg)^{\frac 12} \bigg(\sum_{k\in\mathbb Z \cup \{ 0\}}\sum_{\xi \in \mathcal{F}_k} \nu^2\Big( \frac{\log |\xi| - \log c}{\log q } \Big) |\widehat f(\xi)|^2\bigg)^{\frac12}.
\end{align*}
The conclusion follows choosing $\nu(t)=t^{\frac12+\frac\delta2}+1$.

 \subsection{Proof of Theorem \ref{thm:2x2matrix-complete}: \underline{ case $|\det A|>1$}} We want to prove the analogous of Lemma \ref{lem:detA_1} for a matrix $A$ with $|\det A|>1$.  However, we need a preliminary result, which is a special case of \cite[Lemma 3]{Katz}. The proof we provide here for the reader's convenience is essentially the same one as in \cite{Katz} adapted to the case $d=2$.
\begin{lem}\label{Katznelson_lemma} Let $A$ be a $2\times2$ integer matrix  with  a real irrational eigenvalue $\lambda$ and let $V_\lambda$ be its corresponding eigenspace. Then, there exists $C_A > 0 $ such that, for $\xi \in \bZ^2 \setminus \{0\}$,
\[ \vert \xi \vert \dist(\xi ,V_\lambda)  \geq C_A, \]
    where $\dist$ is the Euclidean distance between $\xi$ and $V_\lambda.$
\end{lem}

\begin{proof}
    By Dirichlet's theorem, for every $Q\in \bN$, there exists $q \in \bN,q\leq Q$ and $r\in \bZ$ such that 
    \[ \Big| \lambda - \frac{r}{q} \Big| < \frac{1}{q Q}. \]
Now, fix $\xi \in \bZ^2\setminus \{0 \}$ and notice that $ (q A-r)\xi \in \bZ^2 \setminus \{0\}$, so $ 1/q \leq \vert (A- r/q)\xi \vert $.  Let $y$ be the orthogonal projection of $\xi$ on $V_\lambda$. We have 
\begin{align*}
    \frac{1}{q} \leq & \Big \vert \Big(A-\frac{r}{q}\Big)\xi \Big\vert  = \Big\vert \Big(A-\frac rq \Big) (\xi-y) + \Big(\lambda-\frac rq \Big) y \Big\vert \\
    \leq & (\Vert A \Vert + |\lambda|+1 ) \dist(\xi,V_\lambda) + \frac{\vert \xi \vert}{ q Q }.
\end{align*}
Setting $C = \Vert A \Vert + |\lambda | +1 $ and rearranging the above inequality we get 
\[ \Big( 1-\frac{\vert \xi \vert}{Q} \Big) \leq C \dist(\xi ,V_\lambda) q \leq C \dist(\xi ,V_\lambda) Q.  \]
Setting $Q = \lceil 2 \vert \xi \vert\rceil$ we obtain the desired estimate. 
\end{proof}
 \begin{lem}\label{lem:detA_geq_1}.
Let $A$ be a $2\times 2$ integer matrix such that $|\det A|>1$ and no eigenvalue of $A$ is a root of unity. Then, there exist constants $c>0$ and $q>1$ such that, for every $k\in\mathbb N$
\[
\min\{|\xi|:\xi\in A^k\mathbb Z^2\backslash\{0\}\}\geq cq^k.
\]
\end{lem}
\begin{proof}
We study separately the cases when $A$ is diagonalizable and when it is not.  Denote by $\lambda, \Lambda \in \bC$ the eigenvalues of the matrix $A$ so that $|\lambda| \leq |\Lambda|$. Recall that $\det(A)=\lambda \Lambda $ is an integer different from $-1,1,0$. If these eigenvalues are complex, then they are conjugate to each other and $1<|\lambda| = |\Lambda|$.  If the eigenvalues are real, then, either $1<|\lambda|\leq |\Lambda| $ or  $|\lambda|<1<|\Lambda|$. In this last case $\lambda$ and $\Lambda$ cannot be rational, since the characteristic polynomial of $A$ is a monic polynomial with integer coefficients and any rational root of such polynomial is an integer. 
     
     \emph{\underline{A is diagonalizable and $1<|\lambda|\leq |\Lambda|$}}. In this case there exists $S\in \GL_2(\bC)$ such that for every $k\in \bN$
    \[ 
    A^k = S^{-1} \begin{bmatrix}
        \lambda^k & 0 \\ 0 & \Lambda^k
    \end{bmatrix} S.
    \]
     Therefore, for $\xi       \in\mathbb Z^2\setminus \{0\}$,
    \[ 
    \vert A^k \xi \vert \geq  \frac{1}{\Vert S \Vert} \Big|  \begin{bmatrix}
        \lambda^k & 0 \\ 0 & \Lambda^k
    \end{bmatrix} S \xi \Big|  \geq  \frac{ |\lambda  |^{k} |\xi|}{ \Vert S\Vert \, \Vert S^{-1} \Vert} \geq \frac{ |\lambda  |^{k}}{ \Vert S\Vert \, \Vert S^{-1} \Vert} ,
 \] 
and the claim is proved in this case. 

\emph{\underline{A is diagonalizable and $|\lambda| <1< |\Lambda|$.}}  Let $V_\lambda, V_\Lambda$ be the $1$-dimensional eigenspaces corresponding to $\lambda $ and $\Lambda$ respectively, let $\theta \in (0,\pi)$ be the angle between them and let $P_\lambda, P_\Lambda$ be the oblique projections with respect to the axes $V_\lambda, V_\Lambda$. Define a new norm in $\bR^2$ as follows, 
 \[ 
 [\xi]_A := \vert P_\lambda \xi \vert + \vert P_\Lambda \xi \vert. 
 \]
 This is of course equivalent to the Euclidean norm of $\bR^2$ up to multiplicative constants which depends on $A$. In what follows $c, C$ denote positive constants which depend only on $A$ and might change from appearance to appearance. Applying now  Lemma \ref{Katznelson_lemma} for some $\xi\in \bZ^2 \setminus \{0\}$, we have 
 \[
 \dist(\xi,V_\lambda) = |\sin(\theta)| \vert P_\Lambda \xi \vert \geq c \vert \xi \vert ^{-1} \geq c [\xi]_A^{-1}.
 \]
Hence,
\[
\vert P_\Lambda \xi \vert \geq c [ \xi ]_A^{-1}.
\]
Writing $\xi = P_\lambda \xi + P_\Lambda \xi $ and applying $A^k$ we obtain $A^k \xi = \lambda ^k P_\lambda \xi + \Lambda^k P_\Lambda \xi $. Hence, 
\begin{align*}
   C \vert A^k \xi \vert \geq [A^kx]_A & = |\lambda|^k \vert P_\lambda \xi \vert + |\Lambda|^k \vert P_\Lambda \xi \vert      \\[0.6em]
   & = |\lambda|^k( [\xi]_A - \vert P_\Lambda \xi\vert ) + |\Lambda|^k \vert P_\Lambda \xi \vert \\[0.6em]
   & = |\lambda|^k [\xi]_A + (|\Lambda|^k-|\lambda|^k) \vert P_\Lambda \xi \vert \\[0.5em]
  & \geq   |\lambda|^k [\xi]_A + c \frac{|\Lambda|^k-|\lambda|^k}{[\xi]_A} =: f([\xi]_A),
\end{align*}
where $f(t)=|\lambda|^k t + c(|\Lambda|^k-|\lambda|^k) t^{-1}, t>0$. Such function $f$ admits a global minimum at $t_{\min}$,
\[
t_{\min} = \sqrt{\frac{c(|\Lambda|^k - |\lambda|^k)}{|\lambda|^k}}, \quad f(t_{\min}) = 2 \sqrt{ c |\lambda|^k (|\Lambda|^k - |\lambda|^k)}. 
\]
For $k$ sufficiently large the estimate $f(t_{\min}) \geq c \sqrt{|\lambda|^k | \Lambda|^k} = c |\det(A)|^\frac k2 $ holds true, and this, combined with the above estimate, proves the claim.

\emph{\underline{A is not diagonalizable}}. In this case we have a single eigenvalue $\lambda$ with $2 \lambda = \tr(A)$ and  $\lambda^2=\det(A)$. Hence, $\tr(A)^2=4\det(A)$. This implies that $\tr(A)$ is an even integer, so that $\lambda \in \mathbb{Z}\setminus \{0\}$. The Jordan decomposition of $A$ guarantees that
\[
A= S^{-1}\begin{bmatrix}
    \lambda & 1 \\
    0 & \lambda
\end{bmatrix} S, 
\]
for some $S\in \GL_2(\bC)$. However, since the columns of $S$ are obtained by solving a homogeneous system of linear equations with integer coefficients, we can assume, without loss of generality, that $S$ has integer entries. Assume for the moment that $\lambda | k$, i.e., there exists $q\in \bZ$ such that $k=q \lambda$. Then, 
 \begin{equation*}
     A^k = S^{-1} \begin{bmatrix}
         \lambda^ k &  k \lambda^{k-1} \\
         0 & \lambda^k
     \end{bmatrix} S  = \lambda^k S^{-1} \begin{bmatrix}
         1 & q \\
         0 & 1
     \end{bmatrix} S.
 \end{equation*}
 Notice that $  U: = \begin{bmatrix}
    1 & q \\ 0 & 1 
 \end{bmatrix}$ is in $GL_2(\bZ)$. Therefore, 
 \begin{equation*}
     A^k \bZ^2 =\lambda ^ k S^{-1} U S \bZ^2 \subseteq \lambda ^ k S^{-1} U \bZ^2 = \lambda ^ k S^{-1} \bZ^2.
 \end{equation*}
    In particular, it follows that 
    \[
    \delta_k:=\min \big\{ \vert y \vert : y \in A^k \bZ^2 \setminus \{ 0 \} \geq \Vert S \Vert^{-1} |\lambda|^k,
    \]
    which proves the claim when $\lambda | k$. In general, let $k\equiv r\mod |\lambda|, 0\leq r < |\lambda|$. Then, $\delta_n \geq \delta_{k-r} \geq c |\lambda|^{k-r} \geq (c |\lambda|^{-|\lambda|}) |\det(A)|^{\frac {k}{2}}$, and this concludes the proof for a nondiagonalizable matrix $A$. 
  \end{proof}

 We now conclude the proof of Theorem \ref{thm:2x2matrix-complete} in the case $|\det A|>1$. Notice that $A$ satisfies the hypothesis of Lemma \ref{lem:detA_geq_1} if and only if $A^*$ does. 
By Wold's theorem the unitary part of the operator $T_A f = f\circ A$ acts on the subspace $\bigcap_{k\in\mathbb N\cup\{0\}} T^k_A(L_0^2(\mathbb T^2))$, but this intersection is trivial and this follows at once from the fact that $
\bigcap_{k\in\mathbb N\cup\{0\}} (A^*)^{k}\mathbb Z^2 =\{ 0\} $
since 
\[ 
\min\{ |\xi| : \xi \in (A^*)^{k} \bZ^2 \setminus \{ 0 \} \} \geq c q^k \to + \infty \textrm{ as } k\to+\infty,
\]
by Lemma \ref{lem:detA_geq_1} applied to $A^*$.
Therefore, $T_A$ is a unilateral shift with generating wandering subspace
\[
\cV=L^2_0(\mathbb T^2)\backslash T_A(L^2_0(\mathbb T^2))=\overline{\SPAN}\{e^{2\pi i  \xi \cdot x}\}_{\xi \notin A^*(\mathbb Z^2)}.
\]
Hence, Theorem \ref{thm:maximal-unilateral-general} applies and, in particular, it applies with $\varepsilon(n)=(n+1)^{-\frac12}(\log(2+n))^{-\frac32- \eta}$ for any $\eta>0$. The proof now proceeds as in the case of matrices with determinant $\pm 1$. Set $\mathcal{F}_k :=A^{*k}\mathbb Z^2 \setminus A^{*(k+1)}\mathbb Z^2$. By Lemma \eqref{lem:detA_geq_1} applied to $A^*$, for every positive increasing function $\nu$ and $f$ satisfying \eqref{condition_Fourier}, one has
\begin{align*}
\sum_{k\in\mathbb N \cup \{ 0\} }\|\Pi_k f\|_{L^2}&=\sum_{k\in\mathbb N \cup \{ 0\}}\bigg(\sum_{\xi \in \mathcal F_k }  |\widehat f(\xi)|^2\bigg)^{\frac12}\\
&\leq\bigg(\sum_{k\in\mathbb N \cup \{ 0\}} \nu^{-2}(k)\bigg)^{\frac 12}\bigg(\sum_{k\in\mathbb N \cup \{ 0\}}\nu^2(k)\sum_{\xi \in \mathcal{F}_k}|\widehat f(\xi)|^2\bigg)^{\frac12} \\
&\leq \bigg(\sum_{k\in\mathbb N \cup \{ 0\}} \nu^{-2}(k)\bigg)^{\frac 12} \bigg(\sum_{k\in\mathbb N \cup \{ 0\}}\sum_{\xi \in \mathcal{F}_k} \nu^2\Big( \frac{\log |\xi| - \log c}{\log q } \Big) |\widehat f(\xi)|^2\bigg)^{\frac12}.
\end{align*}
The conclusion follows choosing $\nu(t)=t^{\frac12+\frac\delta2}+1$. 
In only remains to prove Corollary \ref{cor:A}, but this is now immediate.
\begin{proof}[Proof of Corollary \ref{cor:A}]
Notice that, thanks to \eqref{svp}, we can repeat the very same argument of the proof of Theorem \ref{thm:2x2matrix-complete} to obtain the conclusion.
\end{proof}

\section{Concluding remarks}\label{final_remarks}

As mentioned in the introduction, functions satisfying condition \eqref{condition_Fourier} on their Fourier coefficients are, for instance, functions in any fractional Sobolev space. Here is another, more general, sufficient condition in terms of the $L^2$ integral modulus of continuity. 

 \begin{prop}\label{prop:continuity} Let $\omega (f,t) , t>0$, be the
modulus of continuity of the function $f \in L^{2}( 
\mathbb{T}^{d}) $,
\[
\omega (f,t):=\sup_{| y| \leq t}\Big( \int_{%
\mathbb{T}^{d}}| f(x+y)-f(x)| ^{2}dx\Big) ^{\frac{1}{2}}.
\]
Also let $\alpha \geq 0$. Then there exists a constant $c$ independent of $f$ such that 
\[
\sum_{\xi \in \mathbb{Z}^{d}}\log ^{\alpha }( 1+| \xi|)
 | \widehat{f}(\xi)| ^{2}\leq c\sum_{j=0}^{+\infty }( 1+j^{\alpha
}) \omega^{2}( f,2^{-j}) .
\]
\end{prop}
\begin{proof}
One has 
\begin{align*}
\sum_{\xi \in \mathbb{Z}^{d}}\log ^{\alpha }(1+|\xi|) | \widehat{f}(\xi)|^{2}  & \leq c \sum_{j=0}^{+\infty }\log ^{\alpha }( 1+2^{j+1}) 
\sum_{2^{j}\leq |\xi|_\infty <2^{j+1}}| \widehat{f}
(\xi)| ^{2}  .
\end{align*}

It then suffices to show that 
\[
\sum_{2^{j}\leq |\xi|_\infty <2^{j+1}}| \widehat{f}
(\xi)| ^{2}\leq c\ \omega ^{2}( f,2^{-j}), \quad \forall j \geq 0.
\]

This inequality is well-known, but it is easier to give a proof than a
reference. Parseval's identity gives 
\[
\int_{\mathbb{T}^{d}}| f(x+y)-f(x)|
^{2}dx=\sum_{\xi \in \mathbb{Z}^{d}}| e^{2\pi i\xi y}-1|
^{2}| \widehat{f}(\xi)| ^{2}.
\]

Write $\xi =(h,k) $, with $h\in \mathbb{Z}$ and $k\in \mathbb{Z}%
^{d-1}$, and take $y= ( 2^{-j-2},0) $. Then $\xi y=2^{-j-2}h$, and 
\begin{align*}
\sum_{\substack{(h,k) \in \mathbb{Z\times Z}^{d} \\ 2^{j}\leq |h|_\infty <2^{j+1} }}| \widehat{f}(\xi)| ^{2} &
\leq c \sum_{\substack{(h,k) \in \mathbb{Z\times Z}^{d} \\ 2^{j}\leq |h|_\infty <2^{j+1} }}| e^{2\pi i2^{-j-2}h}-1|
^{2}| \widehat{f}(\xi)| ^{2} 
& \leq c \omega ^{2}( f,2^{-j-2}) 
&\leq c \omega ^{2}(f,2^{-j}).
\end{align*}

Iterating for each of the $d$ coordinates of $\xi$, one obtains 
\[
\sum_{2^{j}\leq |\xi| _{\infty }<2^{j+1}}| 
\widehat{f}(\xi)| ^{2}\leq c \omega ^{2}( f,2^{-j}) .
\]
\end{proof}

It is interesting to observe that, using the above proposition, Corollary \ref{cor:A} can be applied whenever $f$ is the characteristic function of a domain with a fractal boundary with a minimally regular geometry. More precisely, let $\Omega \subseteq \bT^d$ be a Borel measurable set, and suppose that there exists $\varepsilon>0$ such that 
\[
|\{x\in \bT^d: \dist(x,\partial \Omega) \leq t  \} | \leq c \big( \log 1/t \big)^{-2-\varepsilon}, \,\,\text{for all} \,\,\, 0<t<1/2.
\]
Notice that this is an assumption on the Minkowski content of $\partial\Omega$. Then, for $0<\delta<\varepsilon$,
\begin{align*}
    \sum_{j=0}^{+\infty} (1+j^{1+\delta})\omega^2(\chi_\Omega,2^{-j}) & = \sum_{j=0}^{+\infty} (1+j^{1+\delta})\sup_{|y| \leq 2^{-j}} \int_{\bT^d}|\chi_{\Omega}(x+y) - \chi_\Omega(x)|^2 dx \\ 
    & \leq \sum_{j=0}^{+\infty} (1+j^{1+\delta}) |\{x\in \bT^d: \dist(x,\partial \Omega) \leq 2^{-j}  \} |  \\ 
    & \leq c \sum_{j=0}^{+\infty}(1+j^{\delta-\varepsilon-1}) < +\infty.
\end{align*}
Hence, by Proposition \ref{prop:continuity}, we have that 
\[ \sum_{\xi \in \bZ^d } \log^{1+\delta}(1+|\xi|)|\widehat{\chi_\Omega}(\xi)|^2 < + \infty. \]

Explicitly, from Corollary \ref{cor:A} we obtain that, for every matrix $A$ satisfying \eqref{svp}, for every $\eta>0$ and for almost every $x\in\mathbb T^d$, there exists $C>0$, depending on $x, \eta, $ and $A$, such that for every $N\in \bN$
 \begin{equation}\label{discrepancy}
   \Big|\frac{1}{N}\sum_{n=0}^{N-1} \chi_{\Omega}(A^n x) - |\Omega| \Big| \leq C N^{-\frac 12 } \log N ^{\frac 32 + \eta}.
   \end{equation}
It is interesting to compare the above estimate with some results in \cite{BCT}.  In particular, in \cite[Corollary 4]{BCT} it is proved that if $\Omega$ is such that for some $\beta>0$ and for every $t>0$, sufficiently small,  one has \[
|\{x\in \bR^d : \dist(x,\partial \Omega)  \leq t \}|\leq ct^{\beta}
\]
and if $\Omega$ satisfies some other mild technical assumptions, then  there exists a constant $c>0$ such that for every distribution of points $\{p_n\}_{n=0}^{N-1}$ there exists an affine copy $\widetilde{\Omega}$ of $\Omega$ which satisfies the estimate
\[
\Big|\frac{1}{N}\sum_{n=0}^{N-1} \chi_{\widetilde{\Omega}}(p_n) - |\tilde{\Omega}| \Big|\geq cN^{-\frac{1}{2}-\frac{\beta}{2}}.
\]
Moreover, in \cite[Theorem 12]{BCT}, it is proved that, under the same hypothesis on $\Omega$, there exists $c>0$ such that for every $N$ there exists a distribution of $N$ points $\{p_n\}_{n=1}^{N}$ such that
\[
\Big|\frac{1}{N}\sum_{n=0}^{N-1} \chi_{\widetilde{\Omega}}(p_n) - |\widetilde{\Omega}| \Big|\leq cN^{-\frac{1}{2}-\frac{\beta}{2}},
\]
for ``many'' affine copies $\widetilde{\Omega}$ of $\Omega$.
We point out that in this last estimate the single set of points $\{p_n\}_{n=0}^{N-1}$ depends on $\Omega$ and $N$, while, in our estimate \eqref{discrepancy} the underlying sequence of points is fixed. However, notice that the numerology in our upper bound \eqref{discrepancy}  tends to coincide, up to a logarithmic transgression, with the upper bound in \cite[Theorem 12]{BCT} when $\beta  $ tends to $0$.

\bibliographystyle{amsalpha}
\bibliography{biblio}
\end{document}